\documentclass[final,3p,times]{elsarticle}

\usepackage{graphics}
\usepackage{amsmath,amsthm}
\usepackage{latexsym}

\usepackage[utf8]{inputenc}
\usepackage[T1]{fontenc}
\usepackage{amssymb}
\usepackage{color}
\usepackage{natbib}
\usepackage{hyperref}
\usepackage{stmaryrd}
\SetSymbolFont{stmry}{bold}{U}{stmry}{m}{n}

 \usepackage[section]{placeins}

\newcommand\figref[1]{Fig.~\ref{#1}}

\newcommand{\Lebesgue}   {{L}}

\newtheorem{theorem}{Theorem}
\newtheorem{lemma}[theorem]{Lemma}
\newtheorem{remark}[theorem]{Remark}
\newtheorem{method}[theorem]{Method}
\newtheorem*{definition*}{Variational Principle.:}

\makeatletter
\def\Eqlfill@{\arrowfill@\Relbar\Relbar\Relbar}
\newcommand{\extendEql}[1][]{\ext@arrow 0359\Eqlfill@{#1}}
\makeatother

\usepackage[usenames,dvipsnames]{xcolor}

\definecolor{FKcolor}{rgb}{0.0705882,0.227451,0.917647}
\definecolor{SMScolor}{rgb}{0.847059,0.278431,0.819608}
\definecolor{HEcolor}{rgb}{0.917647,0.368627,0.0705882}
 
\newcommand\FKout{\bgroup\markoverwith{\textcolor{FKcolor}{\rule[0.5ex]{2pt}{0.4pt}}}\ULon}
\newcommand\SMSout{\bgroup\markoverwith{\textcolor{SMScolor}{\rule[0.5ex]{2pt}{0.4pt}}}\ULon}
\newcommand\HEout{\bgroup\markoverwith{\textcolor{HEcolor}{\rule[0.5ex]{2pt}{0.4pt}}}\ULon}

\journal{Applied Mathematics and Computation}

\begin{document}
%
\begin{frontmatter}
\title{Transparent boundary conditions for a Discontinuous Galerkin Trefftz method}
\author[gsc,math]{Herbert Egger}\ead{eggermathematik.tu-darmstadt.de}
\author[gsc,temf]{Fritz Kretzschmar\corref{cor}}\ead{kretzschmar@gsc.tu-darmstadt.de}
\author[ifc]{Sascha M. Schnepp}\ead{schnepps@ethz.ch}
\author[dece]{Igor Tsukerman}\ead{igor@uakron.edu}
\author[gsc,temf]{Thomas Weiland}\ead{thomas.weiland@temf.tu-darmstadt.de}
\address[gsc]{Graduate School of Computational Engineering, TU Darmstadt, Dolivostrasse 15, 64293 Darmstadt, Germany}
\address[temf]{Institut fuer Theorie Elektromagnetischer Felder, TU Darmstadt, Schlossgartenstrasse 8  , 64289 Darmstadt, Germany}
\address[ifc]{Institute of Geophysics, Department of Earth Sciences, ETH Zurich, CH-8092 Zurich, Switzerland}
\address[math]{AG Numerik und Wissenschaftliches Rechnen, TU Darmstadt, Dolivostrasse 15, 64293 Darmstadt, Germany}
\address[dece]{Department of Electrical \& Computer Engineering, The University of Akron, Akron, Ohio 44325-3904,USA}
\cortext[cor]{Corresponding Author}
%
%
\begin{abstract}
The modeling and simulation of electromagnetic wave propagation is often accompanied by a restriction to bounded domains which requires the introduction of artificial boundaries. The corresponding boundary conditions should be chosen in order to minimize parasitic reflections.
In this paper, we investigate a new type of transparent boundary condition for a discontinuous Galerkin Trefftz finite element method.
The choice of a particular basis consisting of polynomial plane waves allows us to split the electromagnetic 
field into components with a well specified direction of propagation. 
The reflections at the artificial boundaries are then reduced by penalizing components of the field 
incoming into the space-time domain of interest.
We formally introduce this concept, discuss its realization within the discontinuous Galerkin framework, and demonstrate the performance of the resulting approximations by numerical tests. A comparison with 
first order absorbing boundary conditions, that are frequently used in practice, is made.
For a proper choice of basis functions, we observe spectral convergence in our numerical test and an overall dissipative behavior for which we also give some theoretical explanation.
\end{abstract}


\end{frontmatter}

\def\H{\mathbf{H}}
\def\E{\mathbf{E}}
\def\Eup{\mathrm{E}}
\def\Hup{\mathrm{H}}
\def\P{\mathbf{P}}
\def\F{\mathbf{F}}
\def\A{\mathbf{A}}
\def\B{\mathbf{B}}
\def\G{\mathbf{G}}
\def\g{\mathbf{g}}
\def\c{\mathbf{c}}
\def\n{\mathbf{n}}
\def\d{\mathbf{d}}
\def\h{\mathbf{h}}
\def\e{\mathbf{e}}
\def\r{\mathbf{r}}
\def\v{c}
\def\k{\mathbf{k}}
\def\u{\mathbf{v}}
\def\he{\widehat{\e}}
\def\hh{\widehat{\h}}
\def\hd{\widehat{\d}}
\def\Z{Z}
\def\Zinv{Z^{-1}}

\def\TT{\mathbb{T}}
\def\PP{\mathbb{P}}
\def\RR{\mathbb{R}}
\def\VV{\mathbb{V}}

\def\Oh{\Omega_h}
\def\K{K}
\def\f{f}
\def\dK{{\partial \K}}
\def\In{I^n}
\def\F{\mathcal F}
\def\Fih{{\F^{int}_h}}
\def\Fdh{{\F^{bdr}_h}}

\def\ut{t^{n-1}}
\def\ot{t^{n}}
\def\ve{\mathbf{v}^\E}
\def\vh{\mathbf{v}^\H}
\def\be{\mathbf{b}^\E}
\def\bh{\mathbf{b}^\H}
\def\b{\mathbf{b}}
\def\re{\mathbf{r}^\E}
\def\rh{\mathbf{r}^\H}

\def\grad{\mathbf{\nabla}}
\def\curl{\grad \times}
\def\div{\grad \cdot}
\def\dt{\partial_t}
\def\RR{\mathbb{R}}
\def\rot{\text{curl}}
\def\Rot{\text{Curl}}
\def\eps{\epsilon}

\section{Introduction}\label{sec:intro} 
%
We consider the propagation of electromagnetic waves in a domain $\Omega$ filled with a non-conducting dielectric medium. 
In the absence of charges and source currents, the evolution of the electromagnetic fields is governed by the time-dependent Maxwell equations
\begin{align} \label{eq:maxwell}
    \mu \dt  \H + \curl \E = 0 \qquad \text{and} \qquad 
    \eps \dt \E - \curl \H  = 0 \qquad \text{in } \Omega \times \RR_+.
\end{align}
The electric permittivity $\eps$ and the magnetic permeability $\mu$ are assumed to be piecewise constant. 
At time $t=0$ the electric and magnetic fields $\E$ and $\H$ are prescribed by the initial conditions
\begin{align} \label{eq:initial}
\E\left(0\right) = \E_0 \qquad \text{and} \qquad  \H\left(0\right) = \H_0 \qquad \text{in } \Omega. 
\end{align}
If the fields satisfy the constraint conditions $\div \left(\eps \E_0\right)=0$ and $\div \left(\mu \H_0\right)=0$ in the beginning, then 
\begin{align} \label{eq:gauge}
\div \left(\eps \E\right) =0 \qquad \text{and} \qquad \div\left(\mu \H\right)=0 \qquad \text{on } \Omega \times \RR_+,
\end{align}
which follows by taking the divergence in \eqref{eq:maxwell}. 
The two constraint conditions in \eqref{eq:gauge} express the absence of electric charges and magnetic monopoles, respectively.
If the computational domain $\Omega$ is bounded, the system has to be complemented by appropriate boundary conditions. 
We will consider different types of conditions that all can be cast in the general abstract form 
\begin{align} \label{eq:boundary}
    \b \left( \E, \H \right) = \n \times \g \qquad \text{on } \partial \Omega \times \RR_+;
\end{align}
here $\n$ is the outward directed unit normal vector at the domain boundary.

Problems that are described by such a system of equations arise in various applications, for instance, in the modeling of 
optical wave guides \cite{Hadley92} or in antenna design \cite{Milligan05}.
In such cases, boundary conditions of the form 
\begin{align} \label{eq:impedance}
\alpha \, \n \times \E - \beta \, \n \times \left( \H \times \n \right) = \n \times\g,
\end{align}
may be used to model various physically relevant situations, e.g., the presence of perfect electric and magnetic conductors or the action of surface currents describing the emission of energy by an antenna, but also the 
presence of artificial boundaries resulting from a truncation of the domain $\Omega$ which is often introduced to make a simulation feasible.
Following the physical intuition, appropriate boundary conditions at such artificial boundaries should allow waves to leave the domain $\Omega$ without significant reflection.
The \emph{first order absorbing boundary condition}
\begin{align}  \label{eq:silvermueller}
\n \times \E - \Z \, \n \times \left( \H \times \n \right)= 0,
\end{align}
is widely used for this purpose; here $Z=\sqrt{\mu/\eps}$ is the intrinsic impedance of the medium. 
This condition mimics the Silver-M\"uller radiation condition \cite{Muller52,Barucq97,Li14}, 
and it is satisfied exactly by plane waves propagating in the outward normal direction. 
A brief inspection of the Poynting vector 
\begin{align*} 
\n \cdot \P = \n \cdot \left(\E \times \H\right) = \H \cdot \left(\n \times \E\right) = \Z \; |\n \times \H|^2
\end{align*}
reveals that energy is dissipated by transmission through the boundary at every point on the boundary.
We will refer to this condition as {\it first-order absorbing} or {\it Silver-M\"uller} condition throughout the paper.

The simple choice \eqref{eq:silvermueller} can be improved in several ways: 
In \cite{Joly89}, a more accurate absorbing boundary condition is formulated that still involves only first order derivatives of the fields; for a stability analysis, see also \cite{Sonnendrucker10}. 
Other possibilities include the classical Bayliss-Turkel and Enquist-Majda conditions~\cite{Turkel80,Turkel85} and ~\cite{Engquist77,Engquist79}, which allow to systematically construct conditions for arbitrary order. Due to lack of stability, these are however hardly ever used in practice. 
Let us also mention more recent approaches developed by Warburton, Hagstrom, Higdon, and others~\cite{Warburton04,Hagstrom07,Hagstrom98,Higdon86,Higdon87}, 
the pole condition for the Dirichlet-to-Neumann operator, or the use of infinite elements. 
Another strategy to minimize reflections from the artificial boundaries is to add an exterior absorbing layer, in which the fields decay very fast. This approach, known as perfectly matched layers, has been used very successfully in practice~\cite{Berenger94,Berenger96}. The appropriate choice of geometric and physical parameters of the absorbing layer is however not always completely clear in practice, and in some cases it may be necessary to extend the computational domain substantially. 
In principle, it is also possible to formulate exact boundary conditions, e.g., by the coupling to a boundary integral formulation for the exterior domain~\cite{Claes80,Song87}. This treatment leads to boundary conditions that are non-local in space and/or time~\cite{Hiptmair03,Kurz99}, which complicates numerical realization.
For a review and a comparison 
of various kinds of non-absorbing, transparent, or non-reflecting boundary conditions, let us refer to~\cite{zschiedrichphd,deaphd} and the references given there.

\medskip

In this paper, we follow a different strategy for devising local transparent boundary conditions.
The intuition behind our approach is the following:
Motivated by some of the approaches mentioned above, we assume that at any point of the boundary the electromagnetic fields can be expanded into, or at least approximated by, a superposition  
$\left(\E,\H\right) = \sum_j c_j \; \left(\e_j,\h_j\right)$ of plane waves propagating into specific directions $\d_j=\e_j \times \h_j$. 
The three vectors $\d_j$,$\e_j$, and $\h_j$ are assumed to be normalized and orthogonal.
If the wave is not reflected at the boundary, one would expect that
\begin{align} \label{eq:abc}
c_j = 0 \qquad \text{for any direction $\d_j$ with} \quad \n \cdot \d_j < 0.
\end{align}
Incorporating such a condition in an adequate manner into a numerical scheme should therefore 
help to suppress non-physical reflections at artificial boundaries. 
A similar idea has been used previously in the context of finite difference Trefftz schemes~\cite{Tsukerman06,Tsukerman-book07}. 
%
%
To evaluate the stability of such a boundary condition, let us again consider the energy flux
\begin{align*}
\n \cdot \P 
&= \sum_{i^{\prime},j^{\prime}} c_{i^{\prime}} c_{j^{\prime}} \; \n \cdot \left(\e_{i^{\prime}} \times \h_{j^{\prime}}\right)  
 = \sum_{j^{\prime}} c_{j^{\prime}}^2 \; \left(\n \cdot \d_{j^{\prime}}\right) + \sum_{i^{\prime}\ne j^{\prime}}  c_{i^{\prime}} c_{j^{\prime}} \; \n \cdot \left(\e_{i^{\prime}} \times \h_{j^{\prime}}\right)
\end{align*}
across the boundary. 
Note that because of \eqref{eq:abc}, the summation only runs over indices with
$\n \cdot \d_{i^{\prime}} >  0$ and $\n \cdot \d_{j^{\prime}}>0$.
If the wave at the boundary is mainly propagating in one out-going direction, one can 
argue that the last term is dominated by the first term on the right hand side, and 
one obtains outflow of energy over the boundary. 

In order to incorporate a boundary condition related to \eqref{eq:abc} into a numerical method, 
one has to be able to split the approximation of the electromagnetic field locally into plane waves. 
This could be realized within the framework of generalized finite elements~\cite{Babuska92,Melenk96} or 
via Trefftz finite difference approximations~\cite{Tsukerman05,Tsukerman-book07,Tsukerman06}.
Another possibility is provided by the discontinuous Galerkin framework \cite{Reed1973,Cockburn2001,Fezoui2005}, 
which allows one to systematically couple almost arbitrary local approximations 
for the simulation on the global level and to incorporate rather general boundary 
and interface conditions by some sort of penalization.

In this paper, we consider a space-time discontinuous Galerkin framework for Maxwell's equations similar to that introduced in~\cite{Monk14,Lilienthal14}, and we utilize polynomial Trefftz functions for the local approximation which satisfy \eqref{eq:maxwell}-\eqref{eq:gauge} exactly on every element.
This results in a discontinuous Galerkin Trefftz method that has previously been described in (1+1) dimensions~\cite{Kretzschmar14} and later in (3+1) dimensions~\cite{sisc_2014}; 
see also \cite{Petersen:2009id,Farhat14} for a related Trefftz method in acoustics.
%
%
The numerical approximation of partial differential equations by Trefftz functions has been proposed in \cite{trefftz1926} and since then been investigated intensively; see e.g. \cite{Ruge89,Jirousek97,Herrera2000}. 
Since for the problem under investigation, the Trefftz functions depend on space and time, we automatically arrive at a space-time method. Let us refer to \cite{Ruge89,Herrera2000} for a review on the topic and also to \cite{Huttunen07,Badics08,Moiola:2011io} for wave propagation problems in the frequency domain.

One of the basic building blocks of our method is the explicit construction of a basis for the local Trefftz spaces consisting of polynomial plane waves. This allows us to obtain the required local splitting of the discrete electromagnetic fields into plane waves.
The second step consists in formulating a variational form of the absorbing boundary condition \eqref{eq:abc} that can be incorporated within the discontinuous Galerkin framework.
Similar to the realization of other boundary conditions in a discontinuous Galerkin method, the condition \eqref{eq:abc} will be satisfied approximately by some sort of penalization.
%

To illustrate the benefits of our approach, we present numerical tests including a comparison with first order absorbing boundary conditions. In our computations, we observe spectral convergence for a model problem, provided that the propagation directions of the polynomial plane wave basis functions are chosen appropriately. This indicates that the boundary condition may formally be accurate of arbitrary order. With our numerical tests, we also illustrate energy dissipation and thus stability of the absorbing boundary conditions.

%

\medskip 

The outline of the paper is as follows: 
In Section~\ref{sec:method}, we introduce the space-time discontinuous Galerkin framework which is the basis for our numerical method.
In Section~\ref{sec:basis}, we then construct the plane wave basis for the local Trefftz approximation spaces and we sketch the construction for two-dimensional problems underlying our numerical tests. 
The implementation of the new transparent boundary condition is discussed in detail in Section~\ref{sec:boundary},  
and results of numerical tests for two simple test problems are reported in Section~\ref{sec:nr}.
The presentation closes with a short summary.

\section{A space-time Discontinuous Galerkin formulation} \label{sec:method}
%

For the numerical simulation of the initial boundary value problem \eqref{eq:maxwell}--\eqref{eq:boundary}, 
we consider a space-time discontinuous Galerkin method. 
We utilize Trefftz polynomials for the local approximations which, by definition, satisfy Maxwell's equations exactly. 
An appropriate choice of the basis allows us to expand the numerical solution locally into polynomial plane waves and to apply our new transparent boundary condition. 
In this section, we introduce the general framework of the method. The construction of a plane wave basis for the polynomial Trefftz space and incorporation of the boundary conditions will be addressed in the following two sections.

\subsection{Notation}
Let $\Oh = \{\K\}$ be a non-overlapping partition of the domain $\Omega$ into regular elements $\K$, e.g., tetrahedral, parallelepipeds, prisms, etc. 
We denote by $\Fih = \{\f = \partial \K \cap \partial \K', \ \K \ne \K' \in \Oh\}$ the set of element interfaces and by $\Fdh = \{\f = \partial \K \cap \partial\Omega, \ \K \in \Oh\}$ the set of faces on the boundary. 
On an element interface $\f= \partial\K_1 \cap \partial\K_2$, 
any piecewise smooth function $\u \in C\left(\Oh\right)^3$ takes on two values 
$\u_1 = \u|_{\K_1}$ and $\u_2=\u|_{\K_2}$. We then denote by
\begin{align*}
\{\u\} = \frac{1}{2}\left(\u_1 + \u_2\right), 
\qquad 
[\n \times \u] = \n_1 \times \u_1 + \n_2 \times \u_2,
\end{align*}
the {\it average} and the {\it jump of the tangential component} of $\u$ across $f = \partial\K_1 \cap \partial\K_2$, respectively;
here $\n_i$ denotes the outward normal vector on the boundary of the element $\K_i$.
Now let $0=t_0 < t_1 < \ldots$ be a partition of the time axis into 
intervals $\In = [\ut,\ot]$. 
For every space-time element $\K \times \In$ with $\K \in \Oh$, we denote by $\PP_p\left(\K \times \In\right)$ the space of polynomials in four variables 
with order up to $p$. We assume that $\eps$ and $\mu$ are constant on $\K \times \In$, and call
\begin{align}  \label{eq:localtrefftz}
\TT_p\left(\K \times \In\right) = \big{\{} \left(\E,\H\right) \in [\PP_p\left(\K \times \In\right)]^6 : \eps \dt \E- \curl \H = 0, \ \mu \dt \H + \curl \E =0, \ \div\E=0, \ \div \H=0 \big{\}} 
\end{align}
the space of {\it local Trefftz polynomials}; this is the space of vector valued polynomials up to order $p$ satisfying Maxwell's equations \eqref{eq:maxwell} and the constraint conditions \eqref{eq:gauge} exactly on the corresponding space-time element.

\subsection{The space-time DG framework}
For the discretization of the wave propagation problem \eqref{eq:maxwell}--\eqref{eq:boundary}, we consider a space-time discontinuous Galerkin framework in the spirit of~\cite{Monk14,Lilienthal14}, but with different approximation spaces and a particular choice of numerical fluxes. 
On every time slab $\Omega \times I^n$, we approximate the field $\left(\E,\H\right)$ by {\em piecewise polynomial Trefftz functions} in
\begin{align} \label{eq:space}
\TT_p\left(\Oh \times I^n\right) := \big{\{} \left(\E,\H\right) : \Omega \times \In \to \RR^6 : \left(\E,\H\right)|_{\K \times I^n} \in \TT_p\left(\K \times \In\right) \quad \text{for all } \K \in \Oh \big{\}}. 
\end{align}
%
Using these approximation spaces in a space-time discontinuous Galerkin framework of \cite{sisc_2014} yields
\begin{method}[Space-time discontinuous Galerkin Trefftz method] \label{meth:stdg} $ $\\
Set $\E_h^0=\E_0$, $\H_h^0=\H_0$. For $n \ge 1$ find $\left(\E_h^n,\H_h^n\right) \in \TT_p\left(\Oh \times I^n\right)$ such that 
for all $\left(\ve,\vh\right) \in \TT_p\left(\Oh \times I^n\right)$
\begin{align*} 
& \sum_{\K \in \Oh} \int_\K \eps \E_h^n\left(\ut\right) \cdot \ve\left(\ut\right) + \mu \H_h^n \left(\ut\right) \cdot\vh\left(\ut\right) 
+ \sum_{\f \in \Fih} \int_{\f \times \In}  [\n \times \H_h^n] \cdot \{\ve\} - [\n \times \E_h^n] \cdot \{\vh\}  
\\ & \qquad \qquad \qquad \qquad 
  + \sum_{\f \in \Fdh} \int_{\f \times \In} b(\n \times \E_h^n, \n \times \H_h^n;\ve,\vh) \\
&= 
 \sum_{\K \in \Oh} \int_\K \eps \E_h^{n-1}\left(\ut\right) \cdot \ve\left(\ut\right) + \mu  \H_h^{n-1}\left(\ut\right) \cdot\vh\left(\ut\right)   
   + \sum_{\f \in \Fdh} \int_{\f \times \In} r(\n \times \g; \ve,\vh).  
\end{align*}
\end{method}
This scheme amounts to the methods presented in \cite{Kretzschmar14,sisc_2014} with a particular 
choice of numerical fluxes.
Note that, in order to complete the definition of Method~\ref{meth:stdg}, we still have to specify the bilinear and linear terms $b$ and $r$ that account for the boundary conditions. This will be done in Section~\ref{sec:boundary}.
\subsection{Basic properties of the method} \label{sec:properties}

Before we proceed, let us make some general remarks about this numerical scheme; see \cite{sisc_2014} for details and proofs.

(i) Since the approximating functions satisfy Maxwell's equations exactly on every element, 
the formulation only contains spatial and temporal interface terms, which penalize the tangential discontinuity of the fields.

(ii) Assume that the true solution $\left(\E,\H\right)$ of problem \eqref{eq:maxwell}--\eqref{eq:boundary} is sufficiently smooth and that the boundary terms are consistently chosen, 
%
e.g., such that $b(\n \times \E, \n \times \H;\ve,\vh) =  r(\n \times \g;\ve,\vh)$ holds for every point on the boundary.
Under this assumption, the whole method is consistent, i.e., any smooth solution of the problem \eqref{eq:maxwell}--\eqref{eq:boundary} also 
satisfies the discrete variational principle. 
To see this, let us have a closer look onto the discrete variational problem: 
by tangential continuity of the fields, the last term in the first line drops out. 
Due to continuity in time, the first terms of the first and third line cancel, 
whereas the boundary terms cancel by assumption. 
%

(iii) Under mild conditions on the boundary terms, the discrete variational problem for one time slab can be shown to be well-posed. 
Let $h$ denote the spatial mesh-size, $\tau=\ot-\ut$ the size of the time step, and $d$ the spatial dimension.
Then the first two terms, which are symmetric positive definite, scale like $h^d$ while the interface and boundary terms scale like $\tau h^{d-1}$. Therefore, the left hand side of the variational principle defines an 
elliptic bilinear form provided that the time step size is not too large. The smallness condition on $\tau$ can be dropped, 
if the boundary terms are dissipative in nature, which is the case for many relevant conditions; 
see the remark at the end of (iv).

(iv) The following energy identity holds
\begin{align*}
& \frac{1}{2} \big(  \|\eps^{1/2} \E_h^n(\ot)\|^2_\Omega + \|\mu^{1/2} \H_h^n(\ot)\|_\Omega^2 \big) 
=       \frac{1}{2} \big(  \|\eps^{1/2} \E_h^{n-1}(\ut)\|^2_\Omega + \|\mu^{1/2} \H_h^{n-1}(\ut)\|_\Omega^2 \big) \\ 
& \qquad \qquad
  -\frac{1}{2} \|\eps^{1/2} \big( \E_h^n(\ut) - \E_h^{n-1}(\ut)\big) \|^2_\Omega 
  - \frac{1}{2} \|\mu^{1/2} \big(\H_h^{n}(\ut) - \H_h^{n-1}(\ut) \big)\|_\Omega^2 \\
& \qquad \qquad 
  + \int_{\partial\Omega \times \In} r(\n \times \g; \E_h^n,\H_h^n) 
       -  b(\n \times \E_h^n, \n \times \H_h^n;\E_h^n,\H_h^n) 
       - \n \times \E_h^n \cdot \H_h^n. 
\end{align*}
This can be seen from adding a zero term $\sum_K \int_K (\eps \dt \E_h^n - \curl \H^h_n) \cdot \ve + (\mu \dt \H^h_n + \curl \E^h_n)=0$ to the variational principle, testing with $\ve=\E_h^n$ and $\vh=\H_h^n$, 
and some elementary algebraic manipulations; see \cite{sisc_2014} for details and proofs. 
The boundary term with $\n \times \E \cdot \H$ arises from partial integration of one curl operator.
%

%
(v) Assume that $b(\n \times \E_h^n, \n \times \H_h^n;\E_h^n,\H_h^n) + \n \times \E_h^n \cdot \H_h^n \ge 0$, which we call condition (D).
Then the discrete variational problem is well-posed without any restriction on the time step size. 
Condition (D) is in fact valid for various types of boundary conditions; see Section~\ref{sec:boundary} for details.
If additionally $\r = 0$, then the discrete electromagnetic energy defined by $\mathcal{E}_h\left(t^n\right) = \int_\Omega \eps |\E_h^n\left(t_n\right)|^2 + \mu |\H_h^n\left(t_n\right)|^2$ is monotonically decreasing in time.
%
%
We therefore call boundary conditions having the property (D) 
of {\it dissipative} nature. 

For details and proofs and some further properties of the resulting scheme, let us refer to ~\cite{sisc_2014}; similar results for related discontinuous Galerkin methods based on more standard polynomial spaces can be found in~\cite{Monk14,Lilienthal14}.

\subsection{Implementation}\label{sec:numerical-imp}
%

Method~\ref{meth:stdg} yields an implicit time stepping scheme. 
To evolve the discrete solution from time step $n-1$ to time step $n$, one has
to solve a linear system corresponding to the discrete variational problem. 
Let us sketch the basic structure of this system:
After choosing a basis $\{\left(\e_j^n,\h_j^n\right) : j=1,\ldots,J\}$ 
for the piecewise Trefftz space $\TT_p\left(\Oh \times \In\right)$, we can expand 
the approximate solution with respect to this basis into $\left(\E_h^n,\H_h^n\right) = \sum_j c_j^n \left(\e_j^n,\h_j^n\right)$. The discrete 
variational problem of Method~\ref{meth:stdg} is then equivalent to 
the linear system 
$$\A^n \c^n = \B^n \c^{n-1} + \G^n$$
with matrices $\A^n$, $\B^n$, and vector $\G^n$ defined by 
\begin{align*}
\A^n_{ij} &= \sum_{\K \in \Oh} \int_\K \eps \e_j^n\left(\ut\right) \cdot \e_i^n\left(\ut\right) + \mu \h_j^n \left(\ut\right) \cdot\h_i^n\left(\ut\right) 
+ \sum_{\f \in \Fih} \int_{\f \times \In}  [\n \times \h_j^n] \cdot \{\e_i^n\} - [\n \times \e_j^n] \cdot \{\h_i^n\}  
\\ & \qquad \qquad \qquad \qquad \qquad \qquad 
  + \sum_{\f \in \Fdh} \int_{\f \times \In} b(\n \times \e_j^n, \n \times \h_j^n;\e_i^n,\h_i^n) \\  
\B^n_{ij} &=  \sum_{\K \in \Oh} \int_\K \eps \e_j^{n-1}\left(\ut\right) \cdot \e_i^n\left(\ut\right) + \mu  \h_j^{n-1}\left(\ut\right) \cdot\h_i^n\left(\ut\right)   
\\
\G^n_{i} &= \sum_{\f \in \Fdh} \int_{\f \times \In} r(\n \times \g;\e_i^n,\h_i^n).  
\end{align*}
%
According to point (iii) in the discussion of Section~\ref{sec:properties}, 
the matrix $A^n$ is positive definite, 
provided that the time step $\tau$ is not too large in comparison with the mesh size.
For {\it dissipative} boundary conditions, this holds without restriction on the size of the time step.  
In general, $A^n$ will however not be symmetric.

Also note that up to translation of time, the same basis can be used on every time slab. 
Therefore, if the size of the time step is kept constant, i.e., $t^n = t^{n-1}+ \tau$ for all $n$ with some $\tau>0$, then the matrices $\A^n$ and $\B^n$ are independent of $n$. 
This situation is particularly convenient from 
a computational point of view, since a factorization of the matrix $\A=\A^n$ may then be computed once \textit{a-priori}, 
and the update of the coefficient vectors $\c^n$ from step $n-1$ to $n$ only requires 
one matrix-vector multiplication and a forward-backward substitution. 
Even if no factorization of $\A$ is available, the linear system can be solved with acceptable computational effort by some iterative method, as the matrix $\A$ stems from discretization of a linear hyperbolic problem and therefore usually has a moderate condition number. 

\section{A basis for the space of Trefftz polynomials}\label{sec:basis}

For the local approximation of the electromagnetic fields on every space-time element $\K \times \In$, 
we use vector valued polynomials satisfying Maxwell's equations
\eqref{eq:maxwell} and the constraint conditions \eqref{eq:gauge} exactly. 
In this section, we construct a particular basis for this space of Trefftz polynomials consisting 
of polynomial plane waves, and we discuss some basic properties of this construction. 
Since we only consider single elements $\K \times \In$, we assume throughout this 
section that the material parameters $\eps$ and $\mu$ are positive constants.

\subsection{Polynomial plane wave functions}

As a basis for the local Trefftz space on the element 
$K \times \In$, we consider polynomial plane waves of the form 
\begin{align} \label{eq:basis}
 \F_{p,i} \left(\r,t\right) = \left(\begin{array}{c} \he_{p,i} \\ \Zinv \hh_{p,i}\end{array}\right) 
\varphi_{p,i}\left(\r,t\right) \qquad \text{with} \qquad \varphi_{p,i}\left(\r,t\right) = 
\left( \hd_{p,i} \cdot \r - \v \, t\right)^p.
\end{align}
The hat symbols are used to denote constant vectors of unit length. 
Note that the material properties enter explicitly via the intrinsic impedance $Z=\sqrt{\mu / \eps}$ 
and the speed of light $\v = 1/\sqrt{\eps \mu}$.
Therefore, the Trefftz basis naturally adapts to local changes in the material.
%

\begin{lemma} \label{lem:fip}
Assume that either $p=0$ or that $p \ge 1$ and $\he_{p,i}$, $\hh_{p,i}$, $\hd_{p,i}$ are mutually orthogonal with $\hd_{p,i}=\he_{p,i} \times \hh_{p,i}$. 
Then any function of the form $\left(\E,\H\right)=\F_{p,i}$ is a vector valued polynomial of degree $p$ which satisfies Maxwell's equations \eqref{eq:maxwell} and the constraint conditions \eqref{eq:gauge}.
We call $\F_{p,i}$ a {\em polynomial plane wave}.
\end{lemma}
\begin{proof}
The case $p=0$ yields constant functions and the assertion is clear.
Now assume $p\ge 1$:
By definition, the electric field component has the form $\E = \he_{p,i} \left(\hd_{p,i} \cdot \r - \v t\right)^p$.
One can then verify by direct computation that
$\dt \E = -\he_{p,i} p c \left(\hd_{p,i} \cdot \r - \v t\right)^{p-1}$ and $\curl \E = \he_{p,i} \times \hd_{p,i} p \left(\hd_{p,i} \cdot \r - \v t\right)$;
similar expressions are obtained for the magnetic field component. 
Maxwell's equations \eqref{eq:maxwell} then reduce to the algebraic conditions
\begin{align*}
 p \left(\hd_{p,i} \cdot \r - \v t\right)^{p-1} \left(-\he_{p,i} + \hh_{p,i} \times \hd_{p,i}\right) &= 0 & \text{and}&&
p \left(\hd_{p,i} \cdot \r - \v t\right)^{p-1} \left(-\hh_{p,i} - \he_{p,i} \times \hd_{p,i}\right) &= 0.  
\end{align*}
The two equations are satisfied if $\he_{p,i} =  \hh_{p,i} \times \hd_{p,i}$ 
and that $\hh_{p,i} = - \he_{p,i} \times \hd_{p,i}$ which are the assumptions of the Lemma.  
Additionally, we have $\div \E = p \left(\hd_{p,i} \cdot \r - \v t\right)^{p-1} \he_{p,i} \cdot \hd_{p,i}$. Therefore,
also the constraint conditions are satisfied if 
the directions $\he_{p,i}$, $\hh_{p,i}$, $\hd_{p,i}$ are orthogonal. 
\end{proof}

\subsection{The polynomial Trefftz space}

We will now utilize the polynomial plane wave functions introduced in the previous section to define a special basis for the space $\TT_p\left(\K \times \In\right)$
of local Trefftz polynomials. 
\begin{lemma} \label{lem:trefftz3d}
Let $\mu,\eps>0$ be constant on $\K \times \In$.
Then
\begin{enumerate}
 \item $\dim \TT_p\left(\K \times \In\right)=\frac{1}{3}\left(p+1\right)\left(p+2\right)\left(2p+9\right)$.
 \item For $p \ge 0$ there exist $i_p=2\left(p+1\right)\left(p+3\right)$ orthogonal vector triples $\left(\hd_{p,i},\he_{p,i},\hh_{p,i}\right)$, $i=1,\ldots,i_p$ 
of unit length with $\hh_{p,i} = \hd_{p,i} \times \he_{p,i}$ and such that the functions $\F_{p,i}$, $i=1,\ldots,i_p$ are linearly independent.
 \item The functions $\F_{i,k}$, $1 \le i \le i_k$, $0 \le k \le p$, form a basis of $\TT_p\left(\K \times \In\right)$.
\end{enumerate}
\end{lemma}
\begin{proof}[Proof of assertions 1. and 3.]
The first assertion follows from the explicit construction of a basis for the space of divergence free Trefftz polynomials in \cite{sisc_2014}. Note that for polynomials of four variables we have $[\dim \PP_p]^6=\left(p+1\right)\left(p+2\right)\left(p+3\right)\left(p+4\right)/4$. 
The two Maxwell equations give $p \left(p+1\right) \left(p+2\right) \left(p+3\right)/4$ independent conditions. 
Applying the divergence operator to Maxwell's equations yields $\eps \dt \div \E = -\div \curl \H = 0$ and $\mu \dt \div \H = \div \curl \E=0$. The two constraint conditions thus only have to be required at one point in time 
and therefore give additional $p \left(p+1\right) \left(p+2\right)/3$ independent conditions. 
The fact that the two sets of conditions are independent can be seen from the construction in \cite{sisc_2014}, 
and the assertion follows by counting arguments.
%
%
Trefftz polynomials $\F_{p,i}$ with different orders are linearly independent. 
The third assertion then follows from 1. and 2. by counting the dimensions.
\end{proof}
\begin{remark} \rm 
We cannot provide a complete proof for assertion 2. of Lemma~\ref{lem:trefftz3d} yet.
The fact that there exist $i_p$ such linear independent functions can however be verified numerically for any $p$ required in practice. 
An explicit construction, for which we verified the assertion for $p \le 5$, is given in Section~\ref{sec:directions} below.
Note that by construction and Lemma~\ref{lem:fip}, we know that the polynomials $\F_{i,p}$ have order $p$ and are members of $\TT^p$. Moreover, by assertion 1. of the previous lemma, there cannot exist more than $i_p$ linear independent Trefftz polynomials of order $p$. 
\end{remark}

\begin{remark} \rm
According to Lemma~\ref{lem:trefftz3d}, every local Trefftz polynomial can be split into a superposition of 
polynomial plane wave functions $\F_{p,i}$. This is the basic requirement for the 
implementation of the boundary condition \eqref{eq:abc}. If we use the polynomial plane wave basis in the 
implementation of the method, then the required decomposition is readily available.
Let us emphasize that the Trefftz polynomials have coupled electric and magnetic field components, they are functions of space and time, but do not have a tensor-product structure.
\end{remark}

\subsection{The two-dimensional setting} \label{sec:2d}

In cases of translational invariance along one direction, Maxwell's equations can be cast in a quasi two-dimen\-sional form. 
For illustration and later reference, let us consider one such case in more detail:
We assume that the domain and all fields are homogeneous in the $z$-direction and that the 
electric field is polarized in this direction. The electromagnetic fields then have the form
$\H=\left(\Hup_1,\Hup_2,0\right)$ and $\E = \left(0,0,\Eup\right)$ with $\Hup_1$, $\Hup_2$, and $\Eup$ only depending on $x$ and $y$. This is known as the TM mode. The computational domain is $\Omega=\Omega' \times R$ with $\Omega'\subset \RR^2$ being the relevant slice of the three dimensional domain at any fixed $z$ and $R \subset \RR$ some interval. 
According to the symmetry assumption on the fields, we require $\n \times \H=0$ at $\Omega' \times \partial R$.
Note that this setup still describes a truly three-dimensional problem with symmetry in the $z$ direction. 
We denote by $\Oh'=\{\K'\}$ a mesh of the two-dimensional domain $\Omega'$, set $\Oh=\{\K = \K' \times R\}$,
and define 
\begin{align} \label{eq:trefftzTE}
\TT_p'\left(\K \times \In\right)&= \big{\{}\left(\E,\H\right)  \in \TT_p\left(\K \times \In\right) : \E=\left(0,0,\Eup\right), \ \H=\left(\Hup_1,\Hup_2,0\right), \ \text{with} \ \Hup_1,\Hup_2,\Eup \quad \text{independent of } z \big{\}}.
\end{align}
The prime is used here to distinguish this formulation from a fully three-dimensional problem.
The construction of a basis for the polynomial Trefftz space is similar to the general case. 
Here we consider functions of the form 
\begin{align}
\F'_{p,i}\left(\r,t\right) = \left(\begin{array}{c} \he'_{p,i} \\ \Zinv \hh'_{p,i} \end{array}\right) \varphi'_{p,i}\left(\r,t\right)
\qquad \text{with} \qquad \varphi'_{p,i}\left(\r,t\right) = \left(\hd'_{p,t} \cdot \r - \v t\right)^p 
\end{align}
where $\he'_{p,i}=\left(0,0,1\right)$ and $\hd'_{p,i}$, $\hh'_{p,i}$ are orthogonal unit vectors in the  $x$-$y$ plane. Note that under these assumptions the vector $\hh'_{p,i}$ is already fixed by the choice of $\hd'_{p,i}$ and the condition $\hd'_{p,i}=\he'_{p,i} \times \hh'_{p,i}$. 
\begin{lemma} \label{lem:trefftz2d}
Let $\mu,\eps>0$ be constant on $\K \times \In$.  Then
\begin{enumerate}
 \item $\dim \TT'_p\left(\K \times \In\right)=\left(p+1\right)\left(p+3\right)$.
 \item For every $p$ there exist $i_p'=2p+3$ orthogonal vector triples $\left(\hd_{p,i}',\he_{p,i}',\hh_{p,i}'\right)$, $i=1,\ldots,i_p'$ 
 consisting of mutually orthogonal vectors with $\he_{p,i}'=\left(0,0,1\right)$ and $\hh_{p,i}'=\hd_{p,i}' \times \he_{p,i}'$. 
 such that the system of functions $\F_{p,i}'$, $i=1,\ldots,i_p'$ are linearly independent.
 \item The functions $\F'_{i,k}$, $1 \le i \le i_k'$, $0 \le k \le p$ form a basis of $\TT_p'\left(\K \times \In\right)$.
\end{enumerate} 
\end{lemma}
The proof follows by counting arguments as in the three dimensional case. 
A particular choice of directions for the two-dimensional setting will again be given in Section~\ref{sec:directions}. 

\begin{remark} \rm
It suffices to consider the field components $\Hup_1$, $\Hup_2$, and $\Eup$ as functions of $x$, $y$, and $t$ only. One could therefore also utilize the alternative representation
\begin{align} \label{eq:trefftz2d}
\widehat{\TT}_p'\left(\K' \times \In\right) = \{\left(\H',\Eup\right) \in [\PP_p\left(\K' \times \In\right)]^3 : \eps \dt \Eup - \curl \H' = 0, \ \mu \dt \H' + \curl \Eup = 0, \ \div \H'=0 \}. 
\end{align}
The symbol $\curl$ here denotes the vector-to-scalar and scalar-to-vector curl, respectively. 
Note that the constraint condition for $\Eup$ is satisfied automatically since $\Eup$ only depends on $x$ and $y$, and the corresponding field $\E=\left(0,0,\Eup\right)$ points into $z$-direction. 
The space $\widehat{\TT}_p'$ is isomorphic with $\TT'_p$ and the results stated in the previous lemma carry over.
\end{remark}

\subsection{Choice of directions} \label{sec:directions}

To complete the description of the construction of our basis, 
we have to find a proper set of independent directions $\left(\hd_{p,i},\he_{p,i},\hh_{p,i}\right)$. 
Let us discuss now in some detail a particular choice that 
we used to define the polynomial plane wave basis in our numerical experiments. 

\paragraph{Three dimensional setting}
%
For $p=0$ 
we choose six independent constant functions, one for each vector component.
For order $p\ge 1$, we proceed as follows:
\begin{enumerate}
\item We choose $p+1$ distinct numbers $z_m$, $m=0,\ldots,p$, well distributed in the interval $\left(-1,1\right)$, such that 
\begin{align*}
\ldots,z_{p-4} < z_{p-2} < z_p< z_{p-1} < z_{p-3} < \ldots. 
\end{align*}
This ordering results in an adequate distribution of the directions, cf. \cite{Moiola:2011io} and see Figure~\ref{fig:direction-orientation} below.
\item For every $z_m$, we choose $2m+3$ equidistantly spaced 
points on the circle $\{\left(x,y,z_m\right) : x^2+y^2+z_m^2=1\}$. This yields in total $\sum_{m=0}^p \left(2m+3\right) = \left(p+1\right)\left(p+3\right)$ 
different directions  $\hd_{p,2j-1}$, $j=1,\ldots,\left(p+1\right)\left(p+3\right)$. 
\item For every direction $\hd_{p,2j-1}$ we set $\hd_{p,2j}=\hd_{p,2j-1}$ and choose two independent, e.g., mutually orthogonal, 
polarizations $\he_{p,2j-1}$, $\he_{p,2j}$ orthogonal to $\hd_{p,2j-1}$. 
\item For any pair $\left(\hd_{p,i},\he_{p,i}\right)$, $i=1,\ldots,2\left(p+1\right)\left(p+3\right)$ we finally define $\hh_{p,i} = \he_{p,i} \times \hd_{p,i}$. 
\end{enumerate}
A possible choice of the directions $\hd_{p,2j-1}$, $j=1,\ldots,\left(p+1\right)\left(p+3\right)$ is depicted in Figure~\ref{fig:direction-orientation}. Note that the linear independence of the corresponding functions $\F_{p,i}$ 
can always be verified numerically. 
A similar construction of directions has been used in~\cite{moiolaphd} to generate a basis for the time harmonic 
problem. The analysis of~\cite{moiolaphd} shows that even (almost) any random choice of directions $\hd_{p,2j-1}$ will yield a linearly independent system of Trefftz functions.

\begin{figure}[!ht]
\centering
\includegraphics[width=0.9\textwidth]{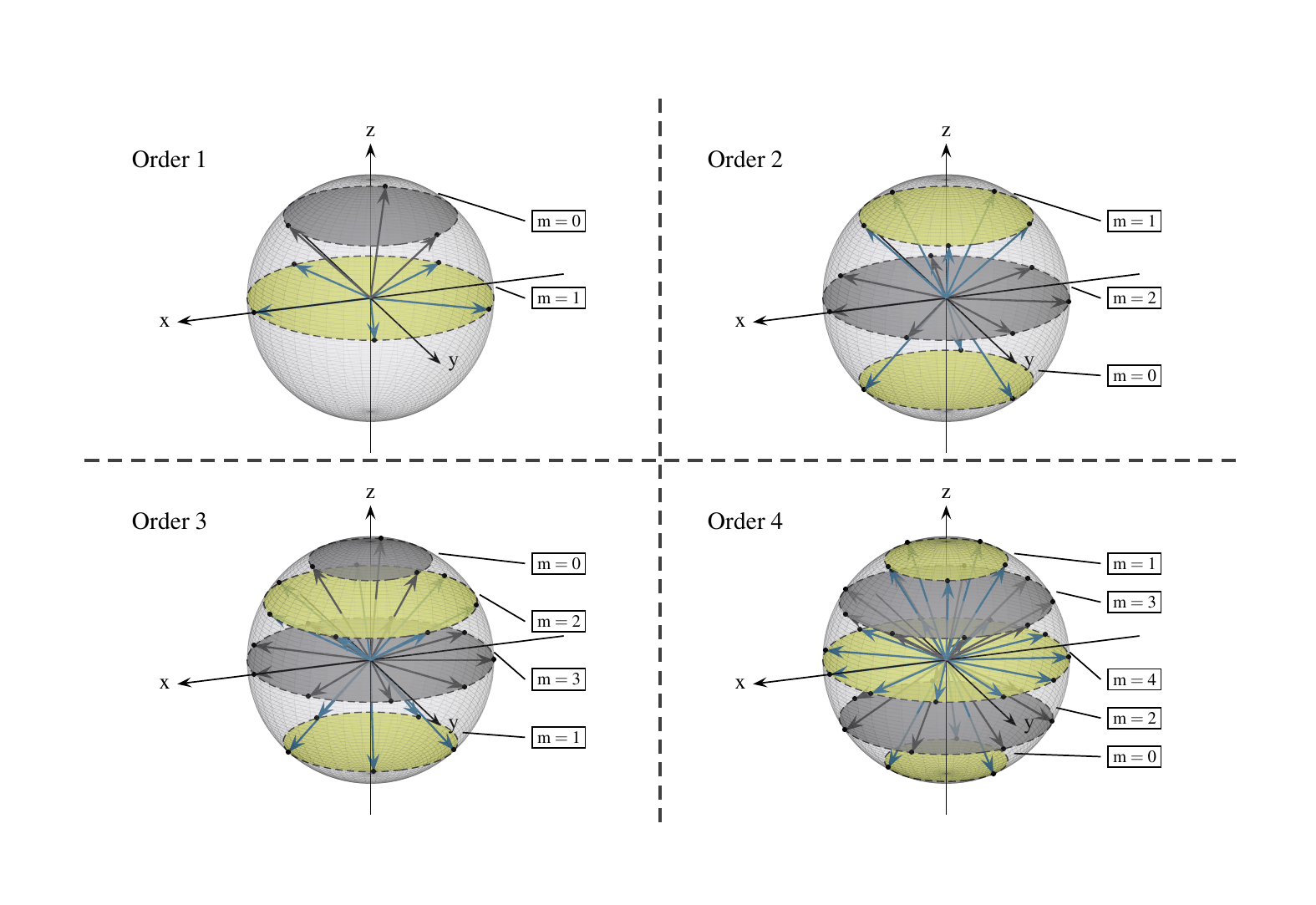}
 \caption[fig:directions]{Choice of directions $\hd_{p,i}$ for $p=1,2,3,4$ in the three-dimensional case. 
   The planes corresponding to $z_m$, $m=0,\ldots,p$ are highlighted in different colors. 
Due to the good distribution and ordering of the levels $z_m$, the directions are distributed 
well over the whole sphere.
 \label{fig:direction-orientation}}
\end{figure}

\paragraph{Two dimensional setting}

For the setting described in Section~\ref{sec:2d}, the construction of a suitable set of directions is much simpler. We choose three constant functions for $p=0$, and proceed for $p>0$ as follows: 
\begin{enumerate}
\item We choose equidistantly spaced directions $\hd_{p,i}$, $i=1,\ldots,2p+3$ on the unit circle $\{\left(x,y,0\right) : x^2+y^2=1\}$. 
\item We define $\he_{p,i}=\left(0,0,1\right)$ and $\hh_{p,i} = \hd_{p,i} \times \he_{p,i}$. 
\end{enumerate}
Linear of the corresponding plane wave functions $\F'_{p,i}$ can again easily be verified numerically.

\section{Incorporation of the boundary conditions} \label{sec:boundary}

To complete the definition of the discontinuous Galerkin method, we
now demonstrate how to incorporate various types of boundary conditions. 
We start by discussing two different implementations for the impedance boundary condition~\eqref{eq:impedance},
which allow us to treat the perfect-electric-conducting (PEC) and perfect-magnetic-conducting (PMC), as well as the first order absorbing Silver-Müller (SM) boundary condition \eqref{eq:silvermueller}. 
Our implementation of the transparent boundary condition~\eqref{eq:abc} will turn out to have a very similar structure. 
In addition to the formulation of these conditions, we also comment on their stability.

\subsection{Representation of PEC-like boundary conditions} \label{sec:pec}

Let us first consider the impedance boundary condition of the form 
\begin{align} \label{eq:pec}
\n \times \E - \beta  (\n \times \H) \times \n = \n \times \g.
\end{align}
For $\beta=0$ and $\g=0$, we arrive at the condition for a perfect electric conductor,
which is why we call conditions of this form PEC-like. 
The choice $\beta=\Z$ and $\g=0$ corresponds to the first-order absorbing boundary condition \eqref{eq:silvermueller}. 
To incorporate conditions of the form \eqref{eq:pec} in Method~\ref{meth:stdg}, 
we choose 
\begin{align*}
b(\n \times \E,\n \times \H;\ve,\vh) = -(\n \times \E) \cdot \vh + \beta (\n \times \H) \cdot (\n \times \vh)
\quad \text{and} \quad 
r(\n \times \g;\ve,\vh) = -(\n \times \g) \cdot \vh.
\end{align*}
This form is consistent with the boundary condition~\eqref{eq:pec}, i.e., $b(\n \times \E,\n \times \H;\ve,\vh)=r(\n \times \g;\ve,\vh)$ holds if \eqref{eq:pec} is valid. 
Let us now consider the energy balance in (iv) of Section~\ref{sec:properties}: 
Testing with $\left(\ve,\vh\right)=\left(\E,\H\right)$ yields 
\begin{align*}
r(\n \times \g;\E,\H) -b(\n \times \E,\n \times \H;\E,\H)  - \n \times \E \cdot \H = 
-(\n \times \g) \cdot \H - \beta  |\n \times \H|^2.
\end{align*}
For $\beta \ge 0$, the last term yields a negative contribution to the energy identity (iv), 
and the resulting method is stable without restriction on the size of the time step.
For $\beta \ge 0$ and $\g = 0$, we thus obtain energy decay on the discrete level.
%


\subsection{Representation of PMC-like boundary conditions} \label{sec:pmc}

Taking the cross product with $\n$ from the right in equation \eqref{eq:pec}, it is possible to obtain an alternative equivalent form of the impedance boundary condition, namely 
\begin{align} \label{eq:pmc}
\n \times \H + \beta' (\n \times \E ) \times \n = \n \times \g'. 
\end{align}
For $\beta'=0$ and $\g'=0$, we arrive at the perfect-magnetic-conducting condition. 
The choice $\beta'=\Zinv$ and $\g'=0$ yields an equivalent form of the first-order absorbing boundary condition \eqref{eq:silvermueller}. 
The condition \eqref{eq:pmc} can be incorporated consistently in the discontinuous Galerkin Trefftz 
method by choosing 
\begin{align*}
b(\n \times \E,\n \times \H; \ve,\vh) = (\n \times \H) \cdot \ve + \beta' (\n \times \E) \cdot (\n \times \ve) 
\quad \text{and} \quad 
r(\n \times \g';\ve,\vh) = (\n \times \g') \cdot \ve.
\end{align*}
Testing with $\left(\ve,\vh\right)=\left(\E,\H\right)$, the boundary term in the energy identity (iv) of Section~\ref{sec:properties} now gives 
\begin{align*}
r(\n \times \g';\E,\H) - b(\n \times \E,\n \times \H; \E,\H) - \n \times \E \cdot \H 
= (\n \times \g') \cdot \E - \beta' |\n \times \E|^2.
\end{align*}
For any $\beta' \ge 0$ and $\g' = 0$, 
we get a negative contribution in the energy identity and thus a dissipative boundary condition.
The resulting discrete variational system is then well-posed without restriction on the size of
time step.
%
\subsection{A first order absorbing boundary condition}
A linear combination of the two conditions \eqref{eq:pec} and \eqref{eq:pmc} 
with $\beta=\Z$, $\beta'=\Zinv$, and $\g=\g'=0$ yields 
\begin{align*}
b(\n \times \E,\n \times \H; \ve,\vh) 
= \frac{1}{2} \Big( (\n \times \H) \cdot \ve + \Zinv  (\n \times \E) \cdot (\n \times \ve) 
-(\n \times \E) \cdot \vh + \Z \; (\n \times \H) \cdot (\n \times \vh) \Big)
\end{align*}
and $r(\n \times \g';\ve,\vh) = 0$,
which serves as implementation of the Silver-Müller boundary conditions \eqref{eq:silvermueller} in our tests.

\subsection{New transparent boundary conditions} \label{sec:tbc}

The basic idea behind our proposal for a transparent boundary condition is
to locally expand the electromagnetic field $\left(\E,\H\right)$ into a superposition of plane waves
\begin{align} \label{eq:split} 
\left(\E,\H\right) = \sum\nolimits_{j} c_j \left(\e_j,\h_j\right),
\end{align}
and then suppress the incoming parts by appropriate penalization. 
Since we are using a basis consisting of plane waves $\F_j = \left(\e_j,\h_j\right)$,
such a decomposition of the discretized fields is readily available. 
For the approximation of the transparent boundary condition \eqref{eq:abc} 
within our discontinuous Galerkin framework, we then consider the choice 
\begin{align}  \label{eq:tbc} 
b(\n \times \E,\n \times \H;\ve\vh) 
&=  \frac{1}{2} \Big( \n \times \H_{in} \cdot \ve_{in} - \n \times \E_{in} \cdot \vh_{in}  + \Zinv (\n \times \E_{in})  \cdot (\n \times \ve_{in}) + \Z \; (\n \times \H_{in}) \cdot (\n \times \vh_{in}) \Big)
\end{align}
and we set $r(\n \times \g;\ve,\vh) = 0$.
Here $\left(\E_{in},\H_{in}\right) = \sum_{j'} c_{j'} \left(\e_{j'},\h_{j'}\right)$ denotes the incoming part of the electromagnetic fields, i.e., summation is done only over indices $j'$ with $\hd_{j'} \cdot \n < 0$.
This condition has a similar form as the Silver-Müller condition stated above,
but only the incoming fields are taken into account. 
%
%
%
%
Let us also examine the energy balance for the new boundary condition: 
Testing with $\left(\ve,\vh\right)=\left(\E,\H\right)$, and assuming $Z=1$ for simplicity, we obtain
\begin{align*}
\n \times \E \cdot \H + b(\n \times \E,\n \times \H; \E,\H)
& = \frac{1}{2} \Big(
   \n \times \E_{out} \cdot \H_{out} - \n \times \H_{out} \cdot \E_{out} 
 + \n \times \E_{out} \cdot \H_{in} - \n \times \H_{out} \cdot \E_{in} 
\\& \qquad \qquad  
 + \n \times \E_{in} \cdot \H_{out} - \n \times \H_{in} \cdot \E_{out}\big) 
 + |\n \times \E_{in}|^2 + |\n \times \H_{in}|^2 \Big).
\end{align*}
Here, $\E_{out}=\E-\E_{in}$ and $\H_{out}=\H-\H_{in}$ denote the out-going field components.
Similar as on the continuous level, we may argue that
the first two terms in the second line will give a positive contribution, if the numerical solution 
is mainly directed into an outward direction. This can be expected to be the case, 
if the continuous solution has this behavior. 
The third and fourth term can then be absorbed into the first two and the last term via a Young's inequality.
In summary, we thus expect a decay of the discrete energy, which is what we actually observe in our numerical tests.

\section{Test problems and numerical results}\label{sec:nr}
 
For numerical validation of the new transparent boundary condition,
we consider two test problems. 
The first test problem studies the propagation of a plane wave. 
In this scenario an analytic solution is available, 
which allows us to conduct a numerical convergence study. 
In the second test problem, we consider the propagation of a cylindrical wave.
We evaluate the effect of the transparent boundary condition on the dissipation 
of energy by comparing the numerical solutions obtained on 
a large domain and on an artificially truncated domain with different choices 
of boundary conditions. 
%

\subsection{Transmission of a Plane Wave}\label{sec:nr:wave}
%
%
We consider a plane wave propagating in direction $\k=(-1,-1,0)/\sqrt{2}$ 
through a homogeneous medium with parameters $\eps=\mu=1$.
The fields $\E=(0,0,\Eup)$ and $\H=(\Hup_1,\Hup_2,0)$ with 
\begin{align}\label{eq:wave}
\Eup= \exp \left( -\left( k_1 x+k_2 y - t + 8 \right)^2 /4 \right), \qquad \Hup_1=k_2 E, \quad \mathrm{and} \quad \Hup_2=-k_1 E,
\end{align}
satisfy Maxwell's equations \eqref{eq:maxwell} and the constraint conditions \eqref{eq:gauge},
and they will serve as the reference solution.
The evolution of the $\E_3=\Eup$ field component of the analytic solution over time is depicted in~\figref{fig:wave}.

\begin{figure}[ht!]
\centering
\includegraphics[width=0.8\textwidth]{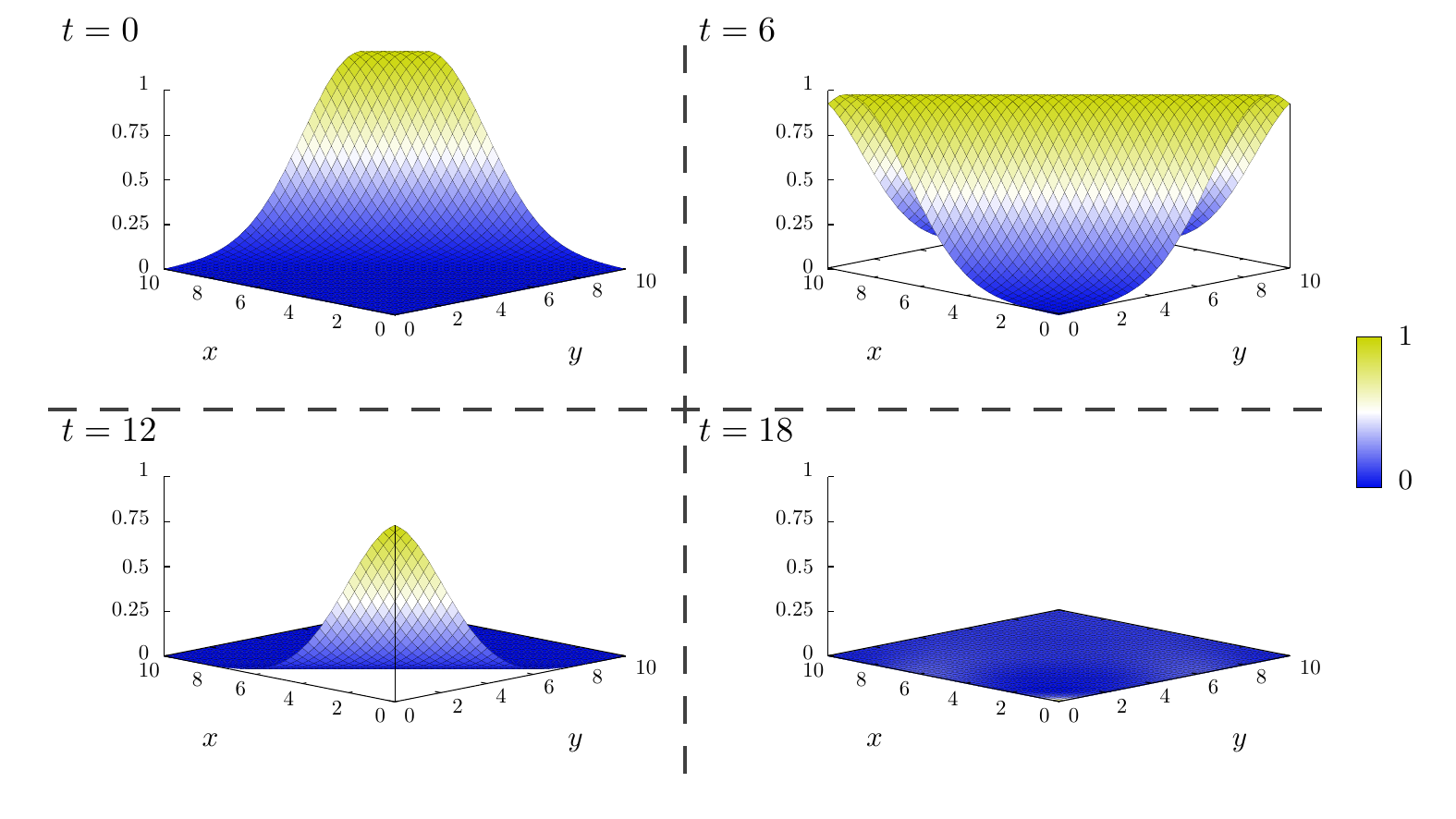}
 \caption[fig:vacuum]{Evolution of the electric field component $\Eup$ of the plane wave 
  \eqref{eq:wave} propagating through the computational domain $\Omega'=(0,10) \times (0,10)$.\\
} \label{fig:wave}
\end{figure}
%
%
Since the fields $\E$ and $\H$ are independent of the third coordinate direction, it suffices to consider 
a geometrically two-dimensional setting; see Section~\ref{sec:2d}.
As a computational domain, we choose $\Omega'=(0,10) \times (0,10)$.
On the incoming boundaries at $(x,10)$ and $(10,y)$, the fields are 
set to that of the analytic solution by the PEC-like boundary condition 
\eqref{eq:pec} with $\beta=1$ and $\g=\E$ determined from the analytic solution.
Different kinds of boundary conditions are utilized at the
boundaries $(x,0)$ and $(0,y)$, where the wave leaves the domain. 
%

%
%
For our simulations, we start from a uniform initial mesh $\Omega_h'$ 
with a mesh size $h=1$ resulting in $N=100$ rectangular elements.
The size of the time step is chosen as $\tau=h/2$ throughout our tests.
We employ Method~\ref{meth:stdg} with approximation spaces $\TT_p'(\Omega_h')$
and different choices of $p$.  
According to our considerations in Section~\ref{sec:2d}, the total number of 
degrees of freedom for one time step is then $N  (p+3)(p+1)$.
Simulations are carried out until $T=24$, where the wave should have left 
the domain almost completely.

In a first series of tests, we evaluate the order of convergence 
with respect to refinement of the spatial and temporal mesh size. 
We run simulations for different approximation orders $p$ and 
on sequences of uniformly refined meshes. In all tests, $\tau=h/2$
is utilized as the time step.
In~\figref{fig:dr_error} we display the relative
error of the computed approximations 
in the $\Lebesgue_2$ space-time norm as a function of the mesh size $h$. 
\begin{figure}[!ht]
\centering
\includegraphics[width=0.65\textwidth]{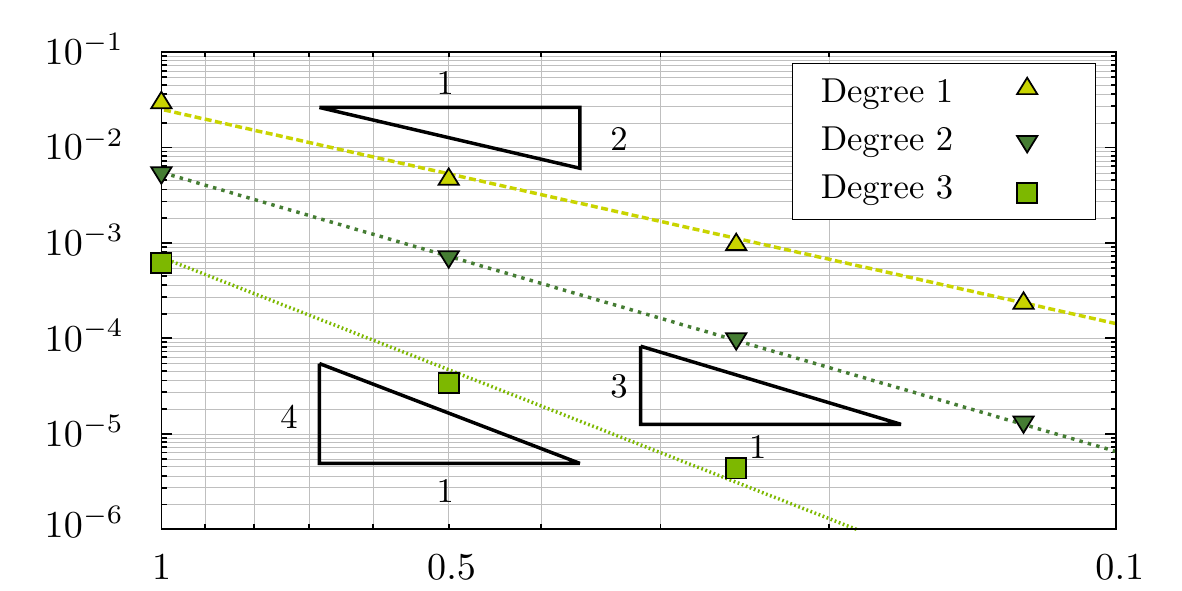}
\caption[fig:dgvsgdt-error]{Relative error $\|\Eup - \Eup_h\|_{L^2(\Omega' \times(0,T))} / \|\Eup\|_{L^2(\Omega' \times(0,T))}$ as a function of the spatial mesh-size $h$ for orders $p=1,2,3$. \\
} \label{fig:dr_error}
\end{figure}
In all simulations, we observe convergence rates of order $p+1$ 
which are optimal with respect to the approximation properties of a piecewise 
polynomial space of order $p$. The discontinuous Galerkin Trefftz method thus
yields a quasi optimal approximation. 

To evaluate more closely the effect of the transparent boundary condition, 
we display in~\figref{fig:order_error} the errors obtained for different choices 
of boundary conditions.
In particular, we compare the implementation of the Silver-Müller condition 
given at the end of Section~\ref{sec:pmc} with the transparent boundary condition discussed in Section~\ref{sec:tbc}. 
Simulations are carried out with polynomial approximation orders
$p=1,2,3$, and $4$ on a uniform mesh with mesh-size $h=1$, $\tau=1/2$, and $N=100$ elements.
\begin{figure}[!ht]
\centering
\includegraphics[width=0.45\textwidth]{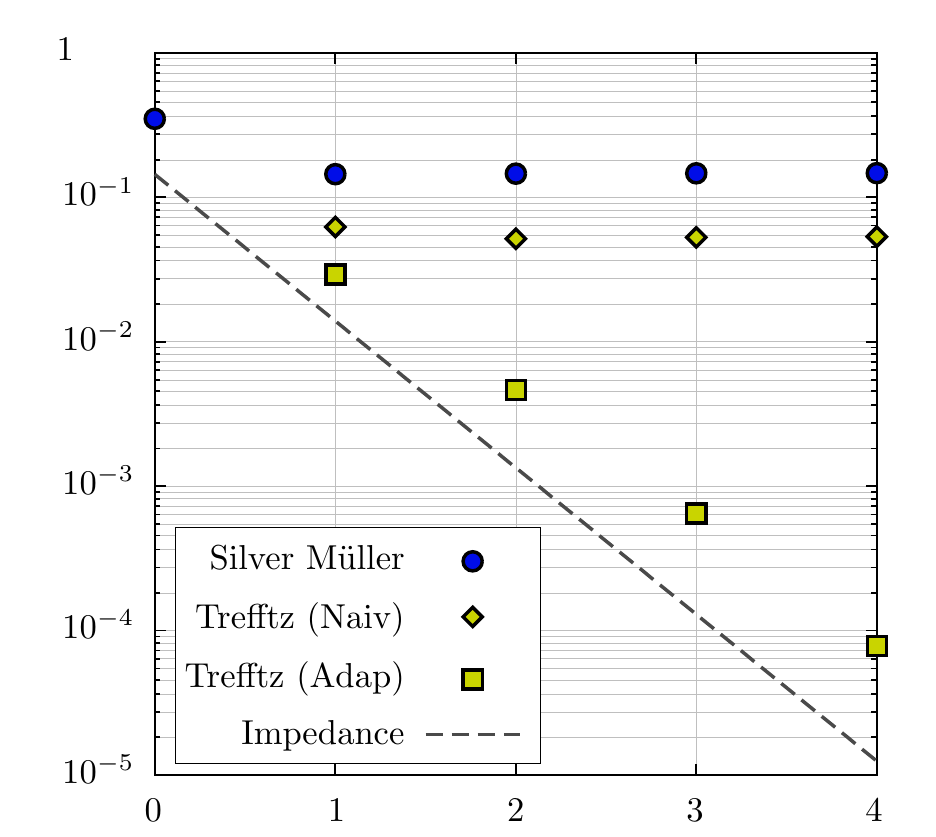}
 \caption[fig:dgvsgdt-error]{Relative error $\|\Eup - \Eup^h\|_{L^2(\Omega' \times(0,T))} / \|\Eup\|_{L^2(\Omega' \times(0,T))}$ versus polynomial degree $p$ for the Silver-M\"uller condition (blue circles), the new transparent boundary conditions, without a direction adaption (green diamonds) as well as with a direction adaption (green squares), and an exact PEC boundary condition (gray line).
} \label{fig:order_error}
\end{figure}
The first-order absorbing condition (blue circles) yields a
saturation due to a systematic consistency error arising from 
the fact that the wave does not impinge on the transparent boundary at normal
angle but at $45^\circ$.
If the essential directions of propagation are not well represented in the basis, the new transparent boundary condition (green diamonds), shows a similar saturation as the first-order absorbing condition.
However, if the essential directions of propagation are well represented in the basis the new transparent boundary condition (green boxes), exhibit spectral convergence.
%
%
%
%
For comparison, we also display (gray dashed line) the results obtained by employing  an exact PEC-like boundary condition \eqref{eq:pec} with $\beta=0$ and $\g=\E$, which may serve as a benchmark for the optimal results that can be expected. 

Let us note that the results obtained with the new transparent boundary condition strongly depend on the choice of directions in the construction of the basis. 
Optimal results are obtained only, if the essential directions of 
propagation of the solution are represented well in the basis. 
This will become obvious also in our second test problem and is in 
accordance with our considerations at the end of Section~\ref{sec:tbc}.
In the simulations above, a direction pointing in the propagation 
direction of the wave was incorporated in the construction of the basis. 
Since the choice of the directions in the construction of the basis 
can be adopted locally at every element to the main direction of propagation, 
the simulation results may still be considered representative.

In summary, we observe that, together with a proper choice of directions in the construction 
of the basis, the new transparent boundary conditions can exhibit exponential convergence.

%
\subsection{Energy dissipation behavior}\label{sec:nr:cylinder}
%

As a second test case, we consider the propagation of wave fields
of the form $\E=(0,0,\Eup)$, $\H=(\Hup_1,\Hup_2,0)$ evolving from the initial
conditions
\begin{align} \label{eq:initial2}
 \Eup^0(x,y)  =  \exp \left( - \left( x^2+ y^2 \right) /18\right), 
\qquad \Hup^0_1=\Hup^0_2=0,
\end{align}
through a homogeneous medium with material parameters $\eps=\mu=1$.
For ease of presentation, we again consider the quasi two-dimensional setting 
discussed in Section~\ref{sec:2d}.
Let us note that a semi-analytic formula for the solution could be obtained 
here via D'Alembert's formula. In the quasi two-dimensional setting, a closed 
form analytic solution is however not available.

As a reference solution, we therefore consider one obtained by numerical simulation 
on a large domain $\widehat \Omega' = (-30,30) \times (-30,30)$. 
For ist construction, we utilize the discontinuous Galerkin
Trefftz method on a uniform mesh with mesh-size $h=1$ and polynomial degree $p=3$.
First order absorbing boundary conditions are prescribed at the outer boundary
$\partial \widehat \Omega'$.  
The evolution of the electric field component $\E_z=\Eup$ is shown in~\figref{fig:pulse}.
\begin{figure}[!!ht]
\centering
\includegraphics[width=0.9\textwidth]{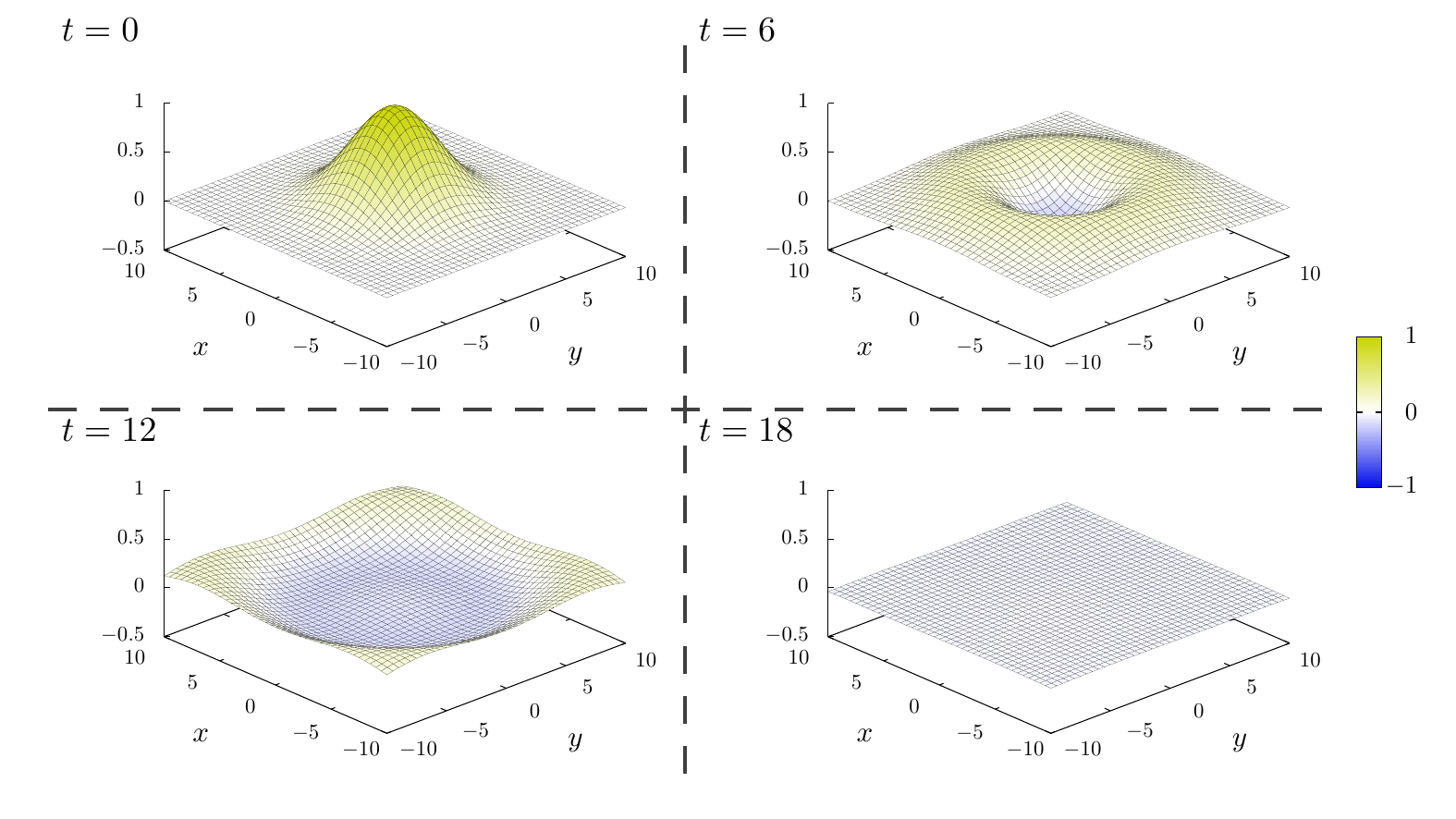}
 \caption[fig:vacuum]{Snapshots of the evolution of the electric field component $\E_z=\Eup$ of a cylindrical wave propagating through a homogeneous medium with parameters $\epsilon=\mu=1$ starting from initial condition \eqref{eq:initial2} depicted in the upper left corner.
} \label{fig:pulse}
\end{figure}
Since the propagation velocity is limited by $c=1/\sqrt{\eps\mu}=1$, 
the boundary condition at the far boundary will have no effect on the solution in the computational domain $\Omega'=(-10,10)\times (-10,10)$ of interest up to time $T=40$.
Note that in contrast to our first test case, the wave does not 
propagate at a fixed angle nor at a fixed velocity now, which can be seen from the two-dimensional D'Alembert formula. 
This will result in an algebraic decay of the energy contained 
in the computational domain $\Omega'$.
For our numerical tests, we consider the artificial restriction of the 
large domain $\widehat \Omega'$ to the computational domain 
$\Omega'=(-10,10) \times (-10,10)$.
Different types of transparent boundary conditions are used at the artificial boundary 
$\partial\Omega'$. 
Note that in this example, the wave front impinges at the artificial boundary at various angles. 
In our simulation, we initially use a uniform mesh with mesh-size $h=1$ resulting in 
$N=400$ rectangular elements. Simulations are conducted for polynomial degree $p=3$ and with time step size $\tau = h/2$ again. 
%
%
%

The evolution of the total energy contained in the computational domain $\Omega'$ is displayed in~\figref{fig:energy}.
\begin{figure}[!!ht]
\centering
\includegraphics[width=0.55\textwidth]{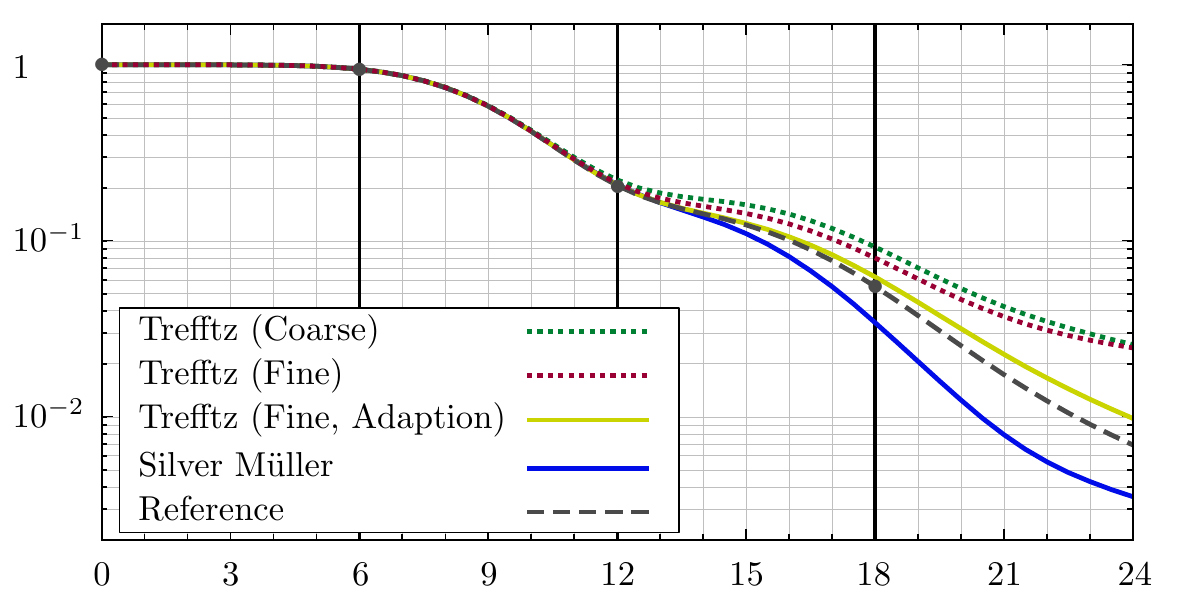}
 \caption[fig:energy]{Evolution of the electromagnetic energy contained in $\Omega'=(0,10)\times(0,10)$ for the reference solution (gray dashed line) and the simulation 
obtained on the truncated domain with transparent boundary conditions (green). 
} \label{fig:energy}
\end{figure}
In the first phase of the simulation, the wave propagates towards the artificial boundary
and the boundary conditions do not have any effect. Through the second phase,
the transparent boundary condition leads to the proper reduction in the energy.  
In the third and fourth phase, the energy is still decaying monotonically. 
The energy decay of the reference solution (gray dashed line), was computed 
on a large domain without artificial boundaries.

We now compare with the results obtained for with different choices of transparent boundary conditions. 
The first simulation (dotted dark green line) was obtained with the new transparent boundary condition \eqref{eq:abc} using the same set of directions in the Trefftz basis on every element. 
In the second case (dotted red line) the simulation was carried out on a finer mesh, i.e. with mesh-size $h=0.5$. 
Both test runs show an energy decay which is slower than that of the reference solution, indicating some amount of artificial reflection.
In the third case (solid green line), a local adaption of the Trefftz basis was employed, i.e., the directions were chosen such that the propagation of the wave front could be represented well. 
As a result, the artificial reflections could be reduced substantially and the energy decay became almost identical to that of the reference solution.
The Silver-Müller boundary condition (solid blue line), on the other hand, has a too dissipative behavior and leads to a much too strong damping of the solution. 
%


In summary, all boundary conditions lead to a decay in energy indicating a dissipative behavior. Together with a proper choice of directions, the new transparent boundary condition lead to the most realistic energy decay.
\section{Summary}\label{sec:summary}
%
In this paper we considered the implementation of a new type of transparent boundary condition in a space-time discontinuous Galerkin method using Trefftz polynomials. The approach is based on a local splitting of the field approximations into a superposition of plane waves and a proper penalization of components corresponding to incoming waves. The required decomposition is available, since we utilize a particular basis for the local Trefftz spaces consisting of polynomial plane wave functions. 
The general procedure is applicable to approximations of arbitrary order, and we observed spectral convergence of the error in our numerical tests, if the directions used in the construction of the polynomial plane wave basis are chosen appropriately. 
Also optimal orders of convergence with respect to the mesh-size were observed in this case. 
While the implementation of the new boundary conditions is similar to that of more standard conditions, like the first order absorbing Silver M\"uller boundary condition, 
the new condition performs substantially better in the our numerical tests.
%
In all cases, the new boundary condition shows a dissipative behavior,
which illustrates the stability of the approach. 
A very realistic energy decay could be obtained by a local adaption of the main directions of the basis functions. 
%
%
\section*{Acknowledgments} 
%
The authors would like to thank the two anonymous reviewers for many helpful suggestions.
The authors were supported by the German Research Foundation (DFG) under grants GSC~233, IRTG 1529, and TRR 154, by the Alexander von Humboldt-Foundation through a Feodor-Lynen research fellowship, and by the National Science Foundation (NSF) under Grant No. 1216927.

%



\end{document}